\documentclass{amsart}

\usepackage{amssymb}
\usepackage{amsthm}
\usepackage{amsmath}
\usepackage[all]{xy}
\usepackage{graphicx}
\usepackage{pifont}
\usepackage{txfonts}

\theoremstyle{plain}
\newtheorem{theorem}{Theorem}[section]
\newtheorem*{theorem*}{Theorem}
\newtheorem*{Main}{Twisted Alexander polynomial with the adjoint action}
\newtheorem{proposition}[theorem]{Proposition}
\newtheorem{lemma}[theorem]{Lemma}
\theoremstyle{definition}
\newtheorem{definition}[theorem]{Definition}
\theoremstyle{remark}
\newtheorem{remark}[theorem]{Remark}

\def\co{\colon\thinspace}
\newcommand{\im}{\mathop{\mathrm{Im}}\nolimits}

\newcommand{\Z}{{\mathbb Z}}

\newcommand{\C}{{\mathbb C}}

\newcommand{\SL}[1][2]{{\mathrm{SL}_{#1}(\C)}}
\newcommand{\GL}{\mathrm{GL}}
\newcommand{\sll}{\mathfrak{sl}_2(\C)}
\newcommand{\matrixE}{\begin{pmatrix} 0 & 1 \\ 0 & 0\end{pmatrix}}
\newcommand{\matrixH}{\begin{pmatrix} 1 & 0 \\ 0 & -1\end{pmatrix}}
\newcommand{\matrixF}{\begin{pmatrix} 0 & 0 \\ 1 & 0\end{pmatrix}}

\newcommand{\trace}{{\rm tr}\,}

\newcommand{\I}{\mathbf{1}}
\newcommand{\bm}[1]{\mbox{\boldmath{$#1$}}}

\newcommand{\proj}{\C \mathrm{P}^1}
\newcommand{\knotexterior}{E_K}
\newcommand{\boundaryTorus}{\partial \knotexterior}
\newcommand{\knotgroup}{\pi_1(E_K)}
\newcommand{\univcover}{\widetilde{\knotexterior}}
\newcommand{\Hyp}{{\rm Hyp}}
\newcommand{\Para}{{\rm Para}}
\newcommand{\twistedAlex}[2]{\Delta_{#1}^{\alpha \otimes #2}(t)}
\newcommand{\twistedAlexAd}[2]{\Delta_{#1}^{\alpha \otimes Ad \circ #2}(t)}
\newcommand{\twistingOp}[1][\cdot]{(-1)^{[#1]}}

\newcommand{\assoc}[1]{{}^a\!#1}

\newcommand{\ie}{i.e.,\,}
\begin{document}


\title{On the twisted Alexander polynomial for metabelian representations into $\SL$}

\author{Yoshikazu Yamaguchi}

\address{Department of Mathematics,
  Akita University
  1-1 Tegata-Gakuenmachi, Akita, 010-8502, Japan}
\email{shouji@math.akita-u.ac.jp}

\date{\today}



\begin{abstract}
  We observe the twisted Alexander polynomial for
  metabelian representations of knot groups into $\SL$ and study
  relations to the characterizations of metabelian representations in the character
  varieties.  We give a factorization of the twisted Alexander polynomial for
  irreducible metabelian representations with the adjoint action on $\sll$, in which
  the Alexander polynomial and the twisted Alexander polynomial appear as factors.
  We also show several explicit examples.  
\end{abstract}


\maketitle

\section{Introduction}
The purpose of this note is to explain explicit forms of the twisted
Alexander polynomial of knot exteriors $\knotexterior$
for metabelian representations mapping
the knot groups into $\SL$.  The twisted Alexander polynomial is a
refinement of the Alexander polynomial of knots defined by using group
homomorphisms from the knot groups into linear groups as
in~\cite{Lin01,Wada94}.  It is of interest to consider the twisted
Alexander polynomial for linear representations which send knot groups
to non--abelian subgroups in $\SL$.  We can regard irreducible
metabelian representations $\rho$ as the simplest ones which have
the non--abelian images.  
The author and F.~Nagasato have given 
a characterization of metabelian representations in the $\SL$-character varieties 
of knot groups.
This characterization is summarized as 
\begin{enumerate}
\item
  \label{item:prop_fixed_points}
  the conjugacy classes of irreducible metabelian representations
  form the fixed point set under an involution of the character variety of a knot group;
\item
  \label{item:prop_corres_branched}
  every conjugacy class of an irreducible  metabelian representation
  corresponds to a non--trivial abelian representation
  of the fundamental group of the double branched cover over $S^3$.
\end{enumerate}

From the viewpoint of a fixed point under the involution in the
$\SL$-character variety, there exists the induced linear isomorphism
on the tangent spaces of the character variety at the fixed point.
Under the identification of the twisted cohomology group and the
tangent space, this linear isomorphism is given by conjugation
between irreducible metabelian representations.  To see this, we will
make a decomposition of $Ad \circ \rho$ into the direct sum of a
$1$-dimensional representation and a $2$-dimensional one. 
The $2$-dimensional direct summand is
defined by another irreducible metabelian representation $\assoc{\rho}$.
The twisted homology group defined by $\assoc{\rho}$ gives the tangent space with the linear transformation of the order $2$,
namely the scalar multiplication by $-1$.

From the decomposition of $Ad \circ \rho$, 
the twisted Alexander polynomial $\twistedAlexAd{\knotexterior}{\rho}$ turns into the product of 
the rational function $\Delta_K(-t) / (-t-1)$ and
the twisted Alexander polynomial
$\Delta^{\alpha \otimes \, \assoc{\rho}}_{\knotexterior} \big( \big(\! \sqrt{-1}\,\big) t \big)$
where $\Delta_K(t)$ is the Alexander polynomial of $K$.
The property~\eqref{item:prop_fixed_points} that the conjugacy class of $\rho$ is fixed by the involution 
impliess the symmetry of the factor $\twistedAlex{\knotexterior}{\assoc{\rho}}$,
which is expressed as
$\twistedAlex{\knotexterior}{\assoc{\rho}} = \Delta_{\knotexterior}^{\alpha \otimes \assoc{\rho}}(-t)$.
Moreover we can show that $\twistedAlex{\knotexterior}{\assoc{\rho}}$ has
only even degree terms.  This is derived from the
property~\eqref{item:prop_corres_branched} for irreducible metabelian
representations and the work of P.~Kirk, C.~Livingston and
C.~Herald~\cite{HeraldKirkLivingston2010}.

Finally, we will touch on the divisibility of the twisted Alexander polynomial for
$Ad \circ \rho$.
From~\cite{HeraldKirkLivingston2010}, if a knot $K$ is slice,
then $\twistedAlex{\knotexterior}{\assoc{\rho}}$ is a Laurent polynomial 
$a f(t^2) \bar{f}(t^{-2}) (t^2+1)$ where $a$ is a complex number and $f$ is a Laurent polynomial
with complex coefficients.
In this case, the factor $t^2+1$ makes
the twisted Alexander polynomial $\twistedAlexAd{\knotexterior}{\rho}$
be a Laurent polynomial.
However we can see several examples 
for non--slice knots $K$
and irreducible metabelian representations $\rho$
which give that  
the twisted Alexander polynomial $\twistedAlexAd{\knotexterior}{\rho}$ are also Laurent polynomials.
We give a sufficient condition on $\rho$ for $\twistedAlexAd{\knotexterior}{\rho}$ to be 
a Laurent polynomial, which is referred as {\it longitude--regular}.

To summarize, we will show the following results:
\begin{Main}[Theorem~\ref{thm:relation_twistedAlexPolys}]
Suppose that an irreducible metabelian $\SL$-representation  $\rho$ is longitude--regular.
Then the twisted Alexander polynomial
$\twistedAlexAd{\knotexterior}{\rho}$ is expressed as 
$$
  \twistedAlexAd{\knotexterior}{\rho}
  = (t-1) \cdot \Delta_K(-t) \cdot
  \frac{
    \Delta^{\alpha \otimes \, \assoc{\rho}}_{\knotexterior} \big( \big(\! \sqrt{-1}\,\big) t \big)
  }{
    t^2-1
  }.
$$
Moreover the factor 
$\Delta^{\alpha \otimes \, \assoc{\rho}}_{\knotexterior} \big( \big(\! \sqrt{-1}\,\big) t \big) / (t^2-1)$
turns into a Laurent polynomial in which every term has even degree.
\end{Main}

\section{Preliminaries}
\label{section:preliminaries}
\subsection{Review on the twisted Alexander polynomial}
Let $K$ be a knot in the $3$-dimensional sphere $S^3$. We denote
the knot group by $\knotgroup$ where $\knotexterior$ is the knot exterior obtained by
removing an open tubular neighbourhood of $K$ from $S^3$.
In this paper, we adopt the definition of the twisted Alexander polynomial of $\knotexterior$
by using Fox differential calculus,
which is due to M.~Wada~\cite{Wada94}.

We deal with {\it irreducible} $\SL$-representations from
$\knotgroup$ into $\SL$ and we will use the notation $\rho$ to denote them.
Here irreducible means that $\rho(\knotgroup)$ has no proper invariant line in $\C^2$.
We also
consider the compositions of $\SL$-representations with the adjoint
action on the Lie algebra.  The adjoint action is the
conjugation on the Lie algebra $\sll$ by elements in $\SL$:
\begin{align*}
  Ad \co \SL &\to \mathrm{Aut}(\sll) \\
  A & \mapsto Ad_A\co\bm{v} \mapsto A\bm{v}A^{-1}
\end{align*}
where $\sll$ is regarded as a vector space over $\C$.
The Lie algebra $\sll$ is generated by the following trace--free matrices:
\begin{equation}
  \label{eqn:basis_sll}
  E = \matrixE,\quad 
  H = \matrixH, \quad
  F = \matrixF.
\end{equation}
\begin{remark}
  The adjoint action $Ad_A$ has the eigenvalues $\xi^2$, $\xi^{-2}$
  and $1$ when an element $A$ in $\SL$ has eigenvalues $\xi^{\pm 1}$.
\end{remark}

\subsubsection{Definition of the twisted Alexander polynomial}
We review the definition of the twisted Alexander polynomial of a knot
$K$ with a representation of the knot group $\knotgroup$.  It is known
that every knot group $\knotgroup$ is finitely presentable and expressed
as follows:
\begin{equation}
  \label{eqn:GenerealPresentationKnotGroup}
  \knotgroup
  =
  \langle
  g_1, \ldots, g_k \,|\, r_1, \ldots, r_{k-1}
  \rangle.
\end{equation}
We choose an $\SL$-representation $\rho$ of $\knotgroup$ and denote by
$\alpha$ the abelianization homomorphism of $\knotgroup$:
$$
\alpha \co \knotgroup \to
\langle t \rangle = \knotgroup / [\knotgroup, \knotgroup] \simeq H_1(\knotexterior;\Z) 
$$
where $\knotgroup / [\knotgroup, \knotgroup]$ is regarded as a
multiplicative group.  Let $\Phi_\rho$ be the $\Z$-linear extension of
$\alpha \otimes \rho$ on the group ring $\Z [\knotgroup]$ as
\begin{align*}
  \Phi_\rho\co \Z [\knotgroup] &\to \mathrm{M}_2(\C[t^{\pm 1}]) \\
  \sum_i a_i \gamma_i & \mapsto \sum_i a_i \alpha(\gamma_i) \otimes \rho(\gamma_i)
\end{align*}
where $\mathrm{M}_2(\C[t^{\pm 1}]) = \C [t^{\pm 1}] \otimes_{\C}
\mathrm{M}_2(\C)$ and we identify $\mathrm{M}_2(\C[t^{\pm 1}])$ with
the set of matrices whose entries are elements in $\C[t^{\pm 1}]$.
We also denote by $\Phi_{Ad \circ \rho}$ the $\Z$-linear extension of
$\alpha \otimes Ad \circ \rho$.

\begin{definition}
  \label{def:twistedAlexander}
  Let $K$ be a knot and the knot group $\knotgroup$ be represented as
  in~\eqref{eqn:GenerealPresentationKnotGroup}.  We suppose that
  $\alpha(g_l) \not = 1$.  Then for an $\SL$-representation of
  $\knotgroup$, we define the twisted Alexander polynomial
  $\twistedAlex{\knotexterior}{\rho}$ as
  \begin{equation}
    \label{eqn:DefTwistedAlexander}
    \twistedAlex{\knotexterior}{\rho} =
    \frac{%
      \det \left(
      \Phi_\rho \left(
      \frac{\displaystyle \partial r_i}{\displaystyle \partial g_j}
      \right)
      \right)_{\substack{1 \leq i \leq k-1, \\  1 \leq j \leq k, j \not = l}}
    }{%
      \det (\Phi_\rho (g_l - 1))}
  \end{equation}
  where the $\partial r_i / \partial g_j$ is a linear combination of
  words in $g_1, \ldots, g_k$, given by Fox differential of $r_i$ by
  $g_j$.  We also define the twisted Alexander polynomial
  $\twistedAlexAd{\knotexterior}{\rho}$ for the composition $Ad \circ
  \rho$ by $\Phi_{Ad \circ \rho}$
  up to a factor $\pm t^{m}$ where $m \in \Z$.
\end{definition}

\begin{remark}
  The assumption that
  $\alpha(g_l) \not = 1$ is a sufficient condition for the denominator
  to be non--zero.
\end{remark}
\begin{remark}
  The twisted Alexander polynomial of $K$ for an $\SL$-representation
  of $\knotgroup$ does not depend on the choices of a presentation of
  $\knotgroup$.  Moreover the twisted Alexander polynomial has the
  invariance under the conjugation of representations.  
\end{remark}
We refer to~\cite{Kitano, Wada94} for the details on the
well--definedness.  We choose the last generator as $g_l$ in
Definition~\ref{def:twistedAlexander} for our
explicit examples in Section~\ref{section:examples}.

\subsection{Review on metabelian representations of knot groups}
\label{subsection:MetabelianRep}
We are interested in irreducible {\it metabelian} representations.
This Subsection briefly reviews known results about metabelian representations.
\begin{definition}
  A representation $\rho$ is {\it metabelian} if the image
  by $\rho$ of the commutator subgroup $[\knotgroup, \knotgroup]$ is an
  abelian subgroup in $\SL$.
\end{definition}

\begin{remark}
  \label{remark:maximalAbelianGroup}
  It is known that every maximal abelian subgroup  in $\SL$ is conjugate to either
  the abelian subgroup consisting of hyperbolic elements or parabolic ones, \ie
  $$
  \Hyp :=
  \left\{\left.
      \begin{pmatrix}
        z & 0 \\
        0 & z^{-1}
      \end{pmatrix}
      \,\right|\,
    z \in \C \setminus \{0\}
  \right\}
  \quad \hbox{or}\quad
  \Para :=
  \left\{\left.
      \begin{pmatrix}
        \pm 1 & w \\
        0 & \pm 1
      \end{pmatrix}
      \,\right|\,
    w \in \C
  \right\}.
  $$
\end{remark}

X-S.~Lin~\cite{Lin01} introduced a useful
presentation of knot groups to express irreducible metabelian representations
(see Definition~\ref{def:LinPresentation}),
which is referred as a {\it Lin presentation}.
Using  a Lin presentation,
we will see that the composition $Ad \circ \rho$ contains 
another metabelian representation $\rho'$ as a direct summand
in Proposition~\ref{prop:embedding_rep}.
\begin{definition}[Lemma~$2.1$ in~\cite{Lin01}]
  \label{def:LinPresentation}
  We suppose that a free Seifert surface $S$ has genus $g$.
  Let $\pi_1(S^3 \setminus N(S))$ be generated by $x_1, \ldots, x_{2g}$ where 
  each $x_i$ corresponds to the core of $1$-handle in the handlebody $S^3 \setminus N(S)$.
  We denote by $\mu$ a meridian on $\partial \knotexterior$.
  Then the knot group $\knotgroup$ is expressed as
  \begin{equation}
    \label{eqn:def_LinPresentation}
    \knotgroup =
    \langle
    x_1, \ldots, x_{2g}, \mu
    \,|\,
    \mu a_i^+ \mu^{-1} = a_i^-, \,
    i=1, \ldots, 2g
    \rangle
  \end{equation}
  where $a_i^{\pm}$ are words in $x_1, \ldots, x_{2g}$
  obtained by pushing up and pushing down the spine $\lor_{i=0}^{2g}
  a_i$ of $S$ along the normal direction in $N(S)$.  We call the
  presentation in \eqref{eqn:def_LinPresentation} the Lin presentation
  associated with $S$.
\end{definition}
For example, Figure~\ref{fig:SpineTrefoil} shows a free
Seifert surface of the left handed trefoil knot.  The closed loops 
$x_1$ and $x_2$ correspond to the core of $1$-handles in the
complement of this Seifert surface.
The spine $\lor_{i=1}^2a_i$ is illustrated in
Figure~\ref{fig:SpineTrefoil}. 
It is easy to see that $a_1^+$, $a_1^-$, $a_2^+$
and $a_2^-$ are homotopic to $x_1$, $x_1 x_2^{-1}$, $x_2^{-1}x_1$ and
$x_2^{-1}$.

\begin{figure}[!ht]
  \begin{center}
    \includegraphics[scale=.5]{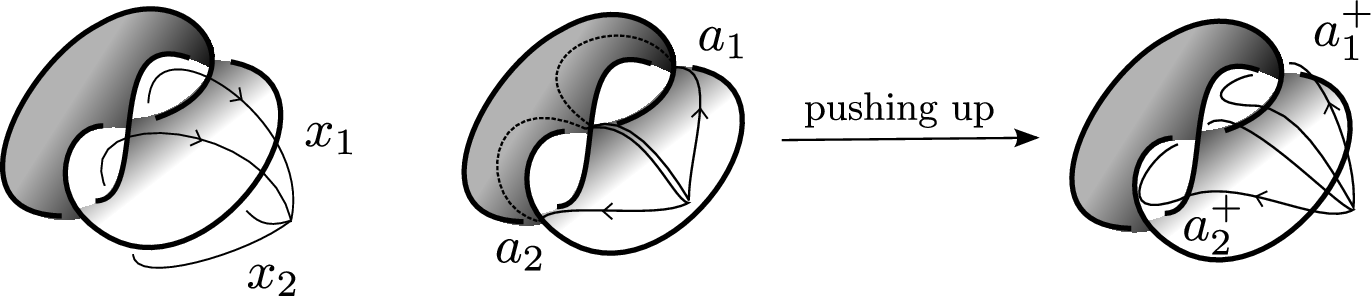}
  \end{center}
  \caption{The spine of the free Seifert surface}\label{fig:SpineTrefoil}
\end{figure}

We see explicit forms of irreducible metabelian representation
via a Lin presentation of $\knotgroup$, which is due to Lin~\cite{Lin01} and 
F.~Nagasato~\cite{nagasato07:_finit_of_section_of_sl}.
\begin{proposition}[Proposition~$1.1$ in~\cite{nagasato07:_finit_of_section_of_sl}]
  \label{prop:explicitIrredMetabelian}
  We fix a Lin presentation of $\knotgroup$ of the form
  $$
  \knotgroup = \langle
  x_1, \ldots, x_{2g}, \mu \,|\,
  \mu a_i^{+} \mu^{-1} = a_i^{-}, i=1, \ldots, 2g
  \rangle.
  $$

  Then 
  every irreducible metabelian $\SL$-representation $\rho$ 
  is conjugate to one given by the following correspondences: 
  \begin{equation}
    \label{eqn:ExplicitCorrespondenceMetabelian}
    x_i \mapsto
    \begin{pmatrix}
      z_i & 0 \\
      0 & z_i^{-1}
    \end{pmatrix},
    \quad
    \mu \mapsto
    \begin{pmatrix}
      0 & 1 \\
      -1 & 0
    \end{pmatrix}
  \end{equation}
  where each $z_i$ is a root of unity whose order is a divisor of $|\Delta_K(-1)|$.
\end{proposition}

In the set of irreducible $\SL$-metabelian representations, 
metabelian representations are characterized as follows.
\begin{proposition}[Proposition~$3$ and Theorem~$1$ in~\cite{NagasatoYamaguchi}]
  \label{prop:fixed_points}
  The number of conjugacy classes of irreducible metabelian
  $\SL$-representations is given by $(|\Delta_K(-1)| - 1)/2$.  These
  conjugacy classes form the fixed point set in the set of conjugacy
  classes of irreducible $\SL$-representations $\rho$ under the
  involution induced by the following correspondence:
  $$
  \rho \mapsto (-\I)^{[\,\cdot\,]} \rho
  $$
  where $(-\I)^{[\,\cdot\,]} \in \{\pm \I\}$ and $[\,\gamma\,]$ denotes
  the homology class in $H_1(\knotexterior;\Z)\simeq \Z$
  for any $\gamma \in \knotgroup$.

  The conjugacy class of such representations $\rho$ correspond bijectively
  to the conjugacy classes of abelian representations for the double branched cover over $S^3$.
\end{proposition}

\section{Explicit forms of the composition $Ad \circ \rho$ for metabelian representations}
\label{section:MetabelianRepresentation} 
We start with an explicit form of the composition of an irreducible
metabelian representation with the adjoint action.
From a Lin presentation of $\knotgroup$,
the composition with the adjoint action can be decomposed into
the direct sum of a $1$-dimensional representation and a $2$-dimensional one.
\begin{proposition}
  \label{prop:explicitIrredMetabelianAd}
  Let $\rho$ be an irreducible metabelian representation of
  $\knotgroup$ into $\SL$, as expressed by using the
  correspondence~\eqref{eqn:ExplicitCorrespondenceMetabelian} in
  Proposition~\ref{prop:explicitIrredMetabelian}.  Then we have the
  following decomposition:
  $$
  Ad \circ \rho = \psi_1 \oplus \psi_2,\quad
  \psi_1 \in \GL(V_1), \quad \psi_2 \in \GL(V_2)
  $$ where $V_1$ and $V_2$ are the subspace $\langle H \rangle$ and $\langle E, F\rangle$ in $\sll$.
  Moreover the
  representations $\psi_1$ and $\psi_2$ are expressed, by taking
  conjugation if necessary, as follows:
  \begin{gather*}
    \psi_1(x_i) = 1, \quad \psi_1(\mu)=-1,\\
    \psi_2(x_i) =
    \begin{pmatrix}
      z_i^2 & 0 \\
      0 & z_i^{-2}
    \end{pmatrix},
    \quad
    \psi_2(\mu) =
    \begin{pmatrix}
      0 & -1 \\
      -1 & 0
    \end{pmatrix}
  \end{gather*}
  for generators of a Lin
  presentation in
  Proposition~\ref{prop:explicitIrredMetabelian}.
\end{proposition}
\begin{proof}
  We can assume, if necessary by taking conjugation, that 
  an irreducible metabelian representation $\rho$ is expressed as 
  $$
  \rho(x_i)
  =\begin{pmatrix}
    z_i & 0 \\
    0 & z_i^{-1}
  \end{pmatrix},
  \quad
  \rho(\mu)
  =\begin{pmatrix}
    0 & 1 \\
    -1 & 0
  \end{pmatrix}
  $$
  for a Lin presentation
  $\knotgroup =
  \langle x_1, \ldots, x_{2g}, \mu \,|\, \mu a_i^+ \mu^{-1} = a_i^-, i=1, \ldots, 2g \rangle$.
  By direct calculation,
  the composition $Ad \circ \rho$ is expressed as 
  $$
  Ad \circ \rho(x_i)
  =\begin{pmatrix}
    z_i^2 & 0 & 0\\
    0 & 1 & 0 \\
    0 & 0 & z_i^{-2}
  \end{pmatrix},
  \quad
  Ad \circ \rho(\mu)
  =\begin{pmatrix}
    0 & 0 & -1 \\
    0 & -1 & 0 \\
    -1 & 0 & 0
  \end{pmatrix}
  $$
  with respect to the basis $\{E, H, F\}$ of $\sll$ as in~\eqref{eqn:basis_sll}.
\end{proof}

\begin{remark}
  \label{remark:conjugation_vector_space}
  An irreducible metabelian representation $\rho$ as in Proposition~\ref{prop:explicitIrredMetabelian}
  is conjugate to $\twistingOp{\rho}$ by the matrix
  $C=\left(\begin{smallmatrix}
    \sqrt{-1} & 0 \\
    0 & -\sqrt{-1}
  \end{smallmatrix}\right)$.
  The conjugation of $C$ acts on $V_2$ as $-\mathrm{id}$ and on $V_1$ as $\mathrm{id}$.
\end{remark}

The $2$-dimensional representation $\psi_2$ is 
related to 
another irreducible metabelian representation into $\SL$.
By taking conjugation of $\psi_2$ by the matrix 
$$D=
\begin{pmatrix}
  e^{3\pi\sqrt{-1}/4} & 0 \\
  0 & e^{-3\pi\sqrt{-1}/4}
\end{pmatrix},
$$
we can see that this conjugate representation gives the following correspondence:
\begin{equation}
\label{eqn:conjugate_psi_2}
\mu \mapsto 
\sqrt{-1}
\begin{pmatrix}
  0 & 1 \\
  -1 & 0
\end{pmatrix}, \quad
x_i \mapsto 
\begin{pmatrix}
  z_i^2 & 0 \\
  0 & z_i^{-2}
\end{pmatrix}
\end{equation}
for a Lin presentation 
$\knotgroup = 
\langle x_1, 
  \ldots, x_{2g}, \mu \,|\, 
  \mu a_i^+ \mu^{-1} = a_i^-, \, i=1, \ldots, 2g 
\rangle$.
From the correspondence~\eqref{eqn:conjugate_psi_2} and
Proposition~\ref{prop:explicitIrredMetabelian}, we can see that 
$D \psi_2 D^{-1}$ is expressed as $\big(\!\sqrt{-1}\,\big)^{[\,\cdot\,]} \rho'$
where $\rho'$ is another irreducible metabelian $\SL$-representation of $\knotgroup$
(see~Proposition~\ref{prop:fixed_points} for the notation
$\big(\!\sqrt{-1}\, \big)^{[\,\cdot\,]}$).  
Hence we have shown the following embedding of irreducible metabelian representation 
$\rho'$ into the composition of another irreducible metabelian representation 
and the adjoint action.
\begin{proposition}
  \label{prop:embedding_rep}
  Let $\rho$ be an irreducible metabelian representation of $\knotgroup$ into $\SL$.
  There exists an irreducible metabelian representation $\rho'$ such that
  $Ad \circ \rho$ is conjugate to the direct sum
  $(-1)^{[\,\cdot\,]} \oplus \big(\!\sqrt{-1}\, \big)^{[\,\cdot\,]} \rho'$
  as an $\SL[3]$-representation.
\end{proposition}

We denote by $\assoc{\rho}$ this
$\SL$-representation $\rho'$ associated with $\rho$.

\section{Computation of the twisted Alexander polynomial}
This section shows several computation results for irreducible metabelian
representations.  The purpose of this section is to provide an
explicit form of the twisted Alexander polynomial for the composition
with the adjoint action.  Moreover we will see a relation of our
result to the twisted Alexander polynomial for the standard action of
$\SL$.

\subsection{Computation for irreducible metabelian representations %
  with and without the adjoint action}
We start with the twisted Alexander polynomial without the adjoint action,
which appears as a factor in $\twistedAlexAd{\knotexterior}{\rho}$.
\begin{proposition}
  \label{prop:TwistedAlexIrrdMetabelian}
  For any irreducible metabelian $\SL$-representation $\rho'$,
  the twisted Alexander polynomial
  $\twistedAlex{E_K}{\rho'}$ is a Laurent polynomial
  which consists of only even degree terms in $t$.
\end{proposition}
\begin{proof}
  We choose a Lin presentation of $\knotgroup$:
  $$\langle x_1, \ldots, x_{2g}, \mu \,|\, \mu a_i^+ \mu^{-1} = a_i^-, i=1, \ldots, 2g \rangle$$
  and suppose that 
  the representation $\rho'$ sends the generators to the following matrices:
  $$
  \rho'(x_i)
  =\begin{pmatrix}
    z_i & 0 \\
    0 & z_i^{-1}
  \end{pmatrix},
  \quad
  \rho'(\mu)
  =\begin{pmatrix}
    0 & 1 \\
    -1 & 0
  \end{pmatrix}.
  $$
  Since some generator $x_i$ is contained in the commutator subgroup and satisfies that $\rho'(x_i) \not = \I$,
  the twisted Alexander polynomial $\twistedAlex{\knotexterior}{\rho'}$
  turns into a Laurent polynomial by Wada's
  criterion~\cite[Proposition~8]{Wada94}. From Proposition~\ref{prop:fixed_points}, 
  the conjugacy class of $\rho'$ corresponds to that of an abelian representation for
  the double branched cover.
  Our invariant coincides with the twisted Alexander polynomial 
  replaced the variable $t$ with $-t^2$ in~\cite[Theorem~7.1]{HeraldKirkLivingston2010}
  (we refer to the formula~$(7.4)$ in~\cite{HeraldKirkLivingston2010}
  for changing variables).
\end{proof}

Next we consider the twisted Alexander polynomial for
the composition of irreducible metabelian
$\SL$-representations with the adjoint action. We assume a technical
condition called \lq\lq longitude--regular\rq\rq\, for
$\SL$-representations.  This condition guarantees the twisted
Alexander polynomial for the composition of $\SL$-representation with
the adjoint action to be a Laurent polynomial.
\begin{theorem}
  \label{thm:twistedAlexanderAdMeta}
  Let an $\SL$-representation $\rho$ be longitude--regular and metabelian.
  Then the twisted Alexander polynomial for the composition of $\rho$ with the adjoint action is expressed as 
  $$
  \twistedAlexAd{E_K}{\rho}
  =
  (t-1)\Delta_K(-t)P(t)
  $$
  where $P(t)$ is a Laurent polynomial satisfying that $P(t)=P(-t)$.
\end{theorem}
\begin{remark}
  The assumption that $\rho$ is longitude--regular includes the
  irreducibility of $\rho$ (see Definition~\ref{def:longitude_regular}).
\end{remark}
\begin{proof}
  We choose a Lin presentation of $\knotgroup$ as
  $$\knotgroup = \langle x_1, \ldots, x_{2g}, \mu \,|\, \mu a_i^+ \mu^{-1} = a_i^-, i=1, \ldots, 2g \rangle.$$
  By Proposition~\ref{prop:explicitIrredMetabelianAd}, 
  we can assume that 
  $Ad \circ \rho = \psi_2 \oplus \psi_1$ such that 
  \begin{gather*}
    \psi_1(x_i) = 1, \quad 
    \psi_1(\mu) = -1,\\
    \psi_2(x_i) =
    \begin{pmatrix}
      z_i^2 & 0 \\
      0 & z_i^{-2}
    \end{pmatrix},\quad
    \psi_2(\mu) =
    \begin{pmatrix}
      0 & -1 \\
      -1 & 0
    \end{pmatrix}.
  \end{gather*}
  The twisted Alexander polynomial
  $\twistedAlexAd{\knotexterior}{\rho}$ for the composition $Ad \circ
  \rho$ is factored into the product of
  $\twistedAlex{\knotexterior}{\psi_1}$ and
  $\twistedAlex{\knotexterior}{\psi_2}$.  From \cite{KL}, it is known
  that the twisted Alexander polynomial
  $\twistedAlex{\knotexterior}{\psi_1}$ is expressed as the rational
  function $\Delta_K(-t)/(-t-1).$ We can see that the twisted
  Alexander polynomial $\twistedAlex{\knotexterior}{\psi_2}$ turns into
  a Laurent polynomial by Wada's
  criterion~\cite[Proposition~8]{Wada94} since there exists a
  commutator $x_i$ such that $\psi_2(x_i) \not = \I$.  Moreover
  $\psi_2$ is conjugate to $\twistingOp{\psi_2}$ by the matrix $C$
  in Remark~\ref{remark:conjugation_vector_space},
  which implies that
  the Laurent polynomial $\twistedAlex{\knotexterior}{\psi_2}$ has the symmetry
  that
  $\twistedAlex{\knotexterior}{\psi_2} =
  \Delta_{\knotexterior}^{\alpha \otimes \psi_2}(-t)$.
  Summarizing the above, we have
  \begin{align}
    \twistedAlexAd{\knotexterior}{\rho}
    &= \twistedAlex{\knotexterior}{\psi_1} \cdot \twistedAlex{\knotexterior}{\psi_2} \nonumber\\
    &= \frac{\Delta_K(-t)}{-t-1} \cdot Q(t) \label{eqn:factorizationTwistedAlex}
  \end{align}
  where $Q(t)$ is a Laurent polynomial satisfying that $Q(t)=Q(-t)$.

  Since $\rho$ is longitude--regular, it follows from~\cite{YY1} that
  the twisted Alexander polynomial
  $\twistedAlexAd{\knotexterior}{\rho}$ has a zero at $t=1$.  It is
  known that $\Delta_K(-1)$ is an odd integer. Hence the Laurent
  polynomial $Q(t)$ has a zero at $t=1$.  Together with the symmetry
  that $Q(t)=Q(-t)$, this implies that we can factor $Q(t)$ into the product
  $(t-1)(t+1)P(t)$.  This factorization of $Q(t)$ completes the proof
  when substituted in~\eqref{eqn:factorizationTwistedAlex}.
\end{proof}

The factorization of Theorem~\ref{thm:twistedAlexanderAdMeta} is deduced
from the decomposition $Ad \circ \rho = \psi_1 \oplus \psi_2$
for an irreducible metabelian representation $\rho$
in~Proposition~\ref{prop:explicitIrredMetabelianAd}.
Furthermore Proposition~\ref{prop:embedding_rep} shows
that the polynomial $P(t)$ is given by the twisted Alexander polynomial for $\assoc{\rho}$.
\begin{theorem}
  \label{thm:relation_twistedAlexPolys}
  For an irreducible metabelian $\SL$-representation $\rho$,
  let $\psi_2$ be the 2-dimensional direct summand in the composition $Ad \circ \rho$.
  Then we have the following equation of the twisted Alexander polynomials:
  $$
  \twistedAlex{\knotexterior}{\psi_2} =
  \Delta_{\knotexterior}^{\alpha \otimes \, \assoc{\rho}}\big( \big(\! \sqrt{-1}\,\big) t \big)
  $$
  And the factor $P(t)$ in Theorem~\ref{thm:twistedAlexanderAdMeta} is given by 
  $\Delta_{\knotexterior}^{\alpha \otimes \, \assoc{\rho}}\big( (\! \sqrt{-1}\,) t \big) / (t^2-1)$,
  \ie
  $$
  \twistedAlexAd{\knotexterior}{\rho}
  = (t-1) \cdot \Delta_K(-t) \cdot 
  \frac{
    \Delta_{\knotexterior}^{\alpha \otimes \, \assoc{\rho}}\big( (\! \sqrt{-1}\,) t \big)
  }{
    t^2-1
  },
  $$
  where $\Delta^{\alpha \otimes \, \assoc{\rho}}_{\knotexterior} \big( \big(\! \sqrt{-1}\,\big) t \big) / (t^2 -1)$
  is a Laurent polynomial which consists of only even degree terms.
\end{theorem}
\begin{proof}
  It follows from 
  $
  \twistedAlex{\knotexterior}{\psi_2}
  = \twistedAlex{\knotexterior}{\big(\!\sqrt{-1}\, \big)^{[\,\cdot\,]} \assoc{\rho}}
  = \Delta_{\knotexterior}^{\big(\!\sqrt{-1}\, \big)^{[\,\cdot\,]} \alpha \otimes \, \assoc{\rho}}( t )
  = \Delta_{\knotexterior}^{\alpha \otimes \, \assoc{\rho}}\big( \big(\!\sqrt{-1}\,\big) t \big)
  $
  and $P(t)$ is $\twistedAlex{\knotexterior}{\psi_2} / (t^2 -1)$.
  Proposition~\ref{prop:TwistedAlexIrrdMetabelian} implies that every term in $\twistedAlex{\knotexterior}{\psi_2}$
  has even degree.
\end{proof}

\subsection{On the twisted homology group and the longitude--regularity}
In this Subsection, we will go into detail on the twisted homology group and the
longitude--regularity of irreducible metabelian $\SL$-representations.  
The longitude--regularity of an $\SL$-representation $\rho$ of
$\knotgroup$ consists of the following conditions on the twisted chain complex 
$C_*(\knotexterior; \sll_\rho)$ given by $Ad \circ \rho$. 

The twisted chain complex $C_*(\knotexterior; \sll_\rho)$ is defined as follows.
We denote by $\univcover$ the universal cover of $\knotexterior$.  The
chain complex $C_*(\univcover;\Z)$ consists of left
$\Z[\knotgroup]$-modules via covering transformation by $\knotgroup$.
Under the action of $\knotgroup$ by $Ad \circ \rho^{-1}$, the Lie
algebra $\sll$ is a right $\Z[\knotgroup]$-module.
\begin{definition}
  By taking the tensor product of $C_*(\univcover; \Z)$ with $\sll$, 
  we have the chain complex $\sll \otimes_{Ad \circ \rho^{-1}} C_*(\univcover; \Z)$.
  This local system is called the twisted chain complex of $\knotexterior$ with the coefficient $\sll_\rho$
  and denoted by $C_*(\knotexterior; \sll_\rho)$.
  We denote by $H_*(\knotexterior;\sll_\rho)$ the homology group of $C_*(\knotexterior; \sll_\rho)$.
\end{definition}

A special case
in the result~\cite{BodenFriedlMetabelianII} of H.~Boden and S.~Friedl 
gives the dimension of $H_i(\knotexterior; \sll_{\rho})$ as follows.

\begin{proposition}[\cite{BodenFriedlMetabelianII}]
  \label{prop:DimLocalSystem}
  For any irreducible metabelian $\SL$-representation $\rho$, we have 
  $\dim_{\C} H_1(\knotexterior;\sll_{\rho}) = 1$.
  Moreover from the irreducibility of $\rho$ 
  and Poincar\'e duality, 
  we can see the dimension of $H_i(\knotexterior; \sll_\rho)$ for all $i$ as follows:
  $$
  \dim_{\C} H_i(\knotexterior; \sll_\rho)=
  \begin{cases}
    1 & i= 1\, \hbox{or}\, 2,\\
    0 & \hbox{otherwise}.
  \end{cases}
  $$
\end{proposition}

We study the homology group $H_*(\knotexterior; \sll_\rho)$ under the involution 
in Proposition~\ref{prop:fixed_points}.  
By Proposition~\ref{prop:explicitIrredMetabelianAd}, The chain complex
$C_*(\knotexterior; \sll_\rho)$ is decomposed into the direct sum of
$V_1 \otimes_{\psi_1^{-1}} C_*(\univcover; \Z)$ and $V_2
\otimes_{\psi_2^{-1}} C_*(\univcover; \Z)$ when 
an irreducible $\SL$-representation
$\rho$ is metabelian.  We denote by
$C_*(\knotexterior; V_i)$ the subchain complex $V_i
\otimes_{\psi_i^{-1}} C_*(\univcover; \Z)$ for $i=1, 2$ and use the notation
$H_*(\knotexterior; V_i)$ for the homology group.
\begin{proposition}
  \label{prop:ReductionLocalSystem}
  If $\SL$-representation $\rho$ is irreducible metabelian, then the
  homology group $H_*(\knotexterior; \sll_\rho)$ is isomorphic to the
  homology group $H_*(\knotexterior; V_2)$.
\end{proposition}
\begin{proof}
  By the decomposition of the chain complex, we have the decomposition
  of the homology group:
  $$
  H_*(\knotexterior; \sll_\rho) \simeq
  H_*(\knotexterior; V_1) \oplus H_*(\knotexterior; V_2)
  $$
  The proof follows from the next Lemma~\ref{lemma:homologyV_1}.
\end{proof}
\begin{lemma}
  \label{lemma:homologyV_1}
  The homology group $H_*(\knotexterior; V_1)$ is trivial.
\end{lemma}
\begin{proof}
  The representation $\psi_1$ is the non--trivial $1$-dimensional representation
  such that a meridian $\mu$ is sent to $-1$.
  Our claim  follows from the computation of example~$2$ in~\cite{KL} and $\Delta_K(-1) \not = 0$.
\end{proof}
\begin{remark}
  The conjugation between $\rho$ and $\twistingOp{\rho}$
  induces the linear isomorphism $-\mathrm{id}$ on $H_1(\knotexterior;V_2)$
  (see Remark~\ref{remark:conjugation_vector_space}).
\end{remark}

The longitude--regularity of $\rho$ is defined concerning a basis of $H_1(\knotexterior;\sll_\rho)$
It was introduced by J.~Porti in~\cite{Porti:1997}. Here we follow the definition in~\cite{JDFourier, YY1}.
\begin{definition}
  \label{def:longitude_regular}
  An $\SL$-representation $\rho$ is longitude--regular if it is
  irreducible and satisfies the following two conditions on the local
  system with the coefficient $\sll_\rho$:
  \begin{enumerate}
  \item $\dim_{\C} H_1(\knotexterior;\sll_\rho) = 1$,
  \item the homomorphism $H_1(\lambda;\sll_{\rho}) \to
    H_1(\knotexterior;\sll_{\rho})$, induced by the inclusion $\lambda
    \to \knotexterior$, is surjective.  Here $\lambda$ is the
    preferred longitude on $\boundaryTorus$ and
    $H_1(\lambda;\sll_{\rho})$ is the homology group for the local
    system of the circle $\lambda$, given by restricting $\rho$ on the subgroup $\langle
    \lambda \rangle \subset \knotgroup$.
  \item If $\rho(\pi_1(\boundaryTorus))$ is contained in $\Para$ by
    conjugation, then $\rho(\lambda) \not = \pm \I$.
  \end{enumerate}
\end{definition}
The conjugation preserves the longitude regularity
of irreducible representations. 
Proposition~\ref{prop:explicitIrredMetabelian} shows that
every irreducible metabelian $\SL$-representation sends meridians to trace
free matrices. Thus we ignore the third condition in
Definition~\ref{def:longitude_regular} for irreducible metabelian
representations.  
Proposition~\ref{prop:DimLocalSystem} shows that
all irreducible metabelian representations satisfy the condition on
the dimension of local system 
in Definition~\ref{def:longitude_regular}.

We give more details on the induced homomorphism from
$H_1(\lambda;\sll_{\rho})$ to the twisted homology group $H_1(\knotexterior;\sll_{\rho})$.
It is known that the homotopy class of the preferred longitude $\lambda$ is
contained in the second commutator subgroup of $\knotgroup$.  If an
$\SL$-representation $\rho$ of $\knotgroup$ is metabelian, the
preferred longitude $\lambda$ is sent to the identity matrix $\I$ by
$\rho$.  Hence we can identify the local system $C_*(\lambda;
\sll_\rho)$ with $\sll \otimes_{\Z} C_*(\lambda; \Z)$.
It is easy to see the twisted homology group for $\lambda$.
\begin{lemma}
  \label{lemma:HomologyLambda}
  If $\rho$ is an irreducible metabelian representation of $\knotgroup$ into
  $\SL$, then the homology group of $C_*(\lambda; \sll_\rho)$ is
  expressed as
  $$
  H_*(\lambda;\sll_\rho) =
  \begin{cases}
    \sll & * = 0\, \hbox{or}\, 1, \\
    \bm{0} & \hbox{otherwise}.
  \end{cases}
  $$
\end{lemma}
\begin{remark}
  \label{remark:ConditionLongitudeRegular}
  This Lemma~\ref{lemma:HomologyLambda} says that an irreducible
  metabelian 
  representation $\rho$ is longitude--regular if and
  only if there exists some $\bm{v} \in \sll$ such that $\bm{v}
  \otimes \lambda$ gives a non--trivial element in
  $H_1(\knotexterior;\sll_\rho)$.
  Here we use the same symbol $\lambda$ for a lift in a universal cover for simplicity of notation.
\end{remark}

By the long exact sequence for the pair $(\knotexterior,
\boundaryTorus)$ and Proposition~\ref{prop:DimLocalSystem} \&
\ref{prop:ReductionLocalSystem}, we have the following short exact
sequence of homology groups with coefficient $V_2$:
$$
0 \to 
H_2(\knotexterior, \boundaryTorus; V_2) \xrightarrow{\delta_*}
H_1(\boundaryTorus; V_2) \xrightarrow{i_*}
H_1(\knotexterior; V_2) \to
0
$$
and 
$$
\dim_\C \ker i_* = 1, \quad  
\dim_\C \im i_* = 1.
$$
By direct calculation, we have a basis of $H_1(\boundaryTorus; V_2)$ as follows:
\begin{lemma}
  Let $\bm{v}_1$ be an eigenvector $E-F$ in $V_2 = \langle E, F \rangle$ for
  the eigenvalue $1$ of $\psi_2(\mu)$.  The homology group
  $H_1(\boundaryTorus; V_2)$ is generated by the homology classes of
  $\bm{v}_1 \otimes \lambda$ and $\bm{v}_1 \otimes \mu$.  Here we
  denote by the same symbols $\lambda$ and $\mu$ a lifts in the
  universal cover.
\end{lemma}

We can write any element of $\ker i_*$ as 
$a [\bm{v}_1 \otimes \lambda] + b [\bm{v}_1 \otimes \mu]$.  
This pair of complex numbers determines
the point $[a:b]$ in the projective space $\proj$.  From
Remark~\ref{remark:ConditionLongitudeRegular}, it follows that an
irreducible metabelian $\SL$-representation $\rho$ is
longitude--regular if and only if $[a:b] \not = [1:0]$.
We can thus express longitude regularity as the condition that the ratio $b/a$ is non--zero.

The ratio $b/a$ can be expressed as a ratio of $1$-forms on the set of 
conjugacy classes of irreducible representations.
For each $\gamma \in \knotgroup$, we define the function $I_\gamma$ on
$\SL$-representations as $I_\gamma(\rho) = \trace \rho(\gamma)$ where
$\rho$ is an $\SL$-representation.  By the invariance of trace under
conjugation, this function $I_\gamma$ also gives a function on the set
of conjugacy classes of $\SL$-representations.  
Then we have the following condition for an irreducible metabelian
$\SL$-representation to be longitude--regular.
We refer to
\cite[Proposition~4.7]{Porti:1997} and its proof for the details.
\begin{proposition}
  Suppose that the function $I_\lambda$ is determined by $I_\mu$ near
  the conjugacy class of an irreducible metabelian
  $\SL$-representation $\rho$.  The ratio $b/a$ satisfies
  \begin{equation}
    \label{eqn:CoeffChange}
    \left(\frac{b}{a}\right)^2 = 
    \frac{I_\mu^2 -4}{I_\lambda^2 -4} 
    \left(
      \frac{d I_\lambda}{d I_\mu}
    \right)^2
  \end{equation}
  Therefore $b/a$ does not vanish at $(I_\lambda, I_\mu) = (2, 0)$ if
  and only if an irreducible metabelian $\SL$-representation $\rho$ is
  longitude--regular.
\end{proposition}
For example, we can see the relation 
$I_\lambda = - (I_\mu^6 - 6 I_\mu^4 + 9 I_\mu^2 -2)$ in the case that $K$ is the trefoil knot.   
The right hand side of~\eqref{eqn:CoeffChange} turns into 
$36$. This means that the ratio $b/a$ is not zero at $(I_\lambda, I_\mu) = (2, 0)$.
Hence every irreducible metabelian $\SL$-representation of the trefoil knot group is longitude--regular.
In the case that $K$ is the figure eight knot,
one can see the relation $I_\lambda = I_\mu^4 -5 I_\mu^2 +2$
(we refer to~\cite[Section~4.5, Example 1]{Porti:1997}).
The right hand side of~\eqref{eqn:CoeffChange} turns into 
$4(2 I_\mu^2 - 5)^2 / (I_\mu^2 - 1)(I_\mu^2 - 5)$.
This means that the ratio $b/a$ is not zero at $(I_\lambda, I_\mu) = (2, 0)$.
Hence every irreducible metabelian $\SL$-representation of the figure eight knot group is longitude--regular.

\section{Examples}
\label{section:examples}
We show four computation examples from Rolfsen's table of knots.  The
first two examples are the trefoil knot and the figure eight knot, for
which the extra factor $P(t)$ in
Theorem~\ref{thm:twistedAlexanderAdMeta} is trivial.  
Two knots where $P(t)$ is non--trivial are the knots $5_1$ and $5_2$.
In each example, the representation $\assoc{\rho}_i$ associated with
$\rho_i$ is given by $\rho_{i+1}$ cyclically.

\subsection{Trefoil knot}
Let $K$ be the trefoil knot.
We consider the free Seifert surface illustrated as in Figure~\ref{fig:SpineTrefoil}.
The Lin presentation associated to this Seifert surface is expressed as 
\begin{equation}
  \label{eqn:LinPresentationTrefoil}
  \knotgroup
  =
  \langle
  x_1, x_2, \mu \,|\,
  \mu x_1 \mu^{-1} = x_1 x_2^{-1},\,
  \mu x_2^{-1} x_1 \mu^{-1} = x_2^{-1}
  \rangle.
\end{equation}

The Alexander polynomial $\Delta_K(t)$ is given by $t^2 -t +1$.  The
knot determinant is $\Delta_K(-1) = 3$.  Since the number of conjugacy
classes of irreducible metabelian representations is given by
$(|\Delta_K(-1)| - 1)/2$, we have only one irreducible metabelian
representation, up to conjugation.  By
Proposition~\ref{prop:TwistedAlexIrrdMetabelian} and the relations in
the Lin presentation~$(\ref{eqn:LinPresentationTrefoil})$, we have the
following representative in the conjugacy class of the irreducible
metabelian representation:
$$
\rho:
x_1 \mapsto
\begin{pmatrix}
  \zeta_3 & 0 \\
  0 & \zeta_3^{-1}
\end{pmatrix},\quad
x_2 \mapsto
\begin{pmatrix}
  \zeta_3^2 & 0 \\
  0 & \zeta_3^{-2}
\end{pmatrix},\quad 
\mu  \mapsto
\begin{pmatrix}
  0 & 1 \\
  -1 & 0
\end{pmatrix}
$$
where $\zeta_3=\exp (2\pi\sqrt{-1}/3)$.
Given that the generator $\mu$ satisfies that $\alpha(\mu)\not = 1$, 
the twisted Alexander polynomials $\twistedAlex{E_K}{\rho}$ and $\twistedAlexAd{E_K}{\rho}$ 
are expressed as
$$
\twistedAlex{E_K}{\rho}
= \frac{
  \det \left( \Phi_\rho \left( \frac{\displaystyle \partial r_i}{\displaystyle \partial x_j} \right) \right)_{\substack{1 \leq i \leq 2, \\  1 \leq j \leq 2}}
}{
  \det (\Phi_\rho (\mu - 1))
},\quad
\twistedAlexAd{E_K}{\rho}
=\frac{
  \det \left( \Phi_{Ad \circ \rho} \left( \frac{\displaystyle \partial r_i}{\displaystyle \partial x_j} \right) \right)_{\substack{1 \leq i \leq 2, \\  1 \leq j \leq 2}}
}{
  \det ( \Phi_{Ad \circ \rho} (\mu - 1))
}
$$
where $r_1 = \mu x_1 \mu^{-1} x_2 x_1^{-1}$ and $r_2 = \mu x_2^{-1} x_1 \mu^{-1} x_2$.
Then we have, up to a factor $\pm t^n$ ($n \in \Z$),
\begin{align*}
  \twistedAlex{E_K}{\rho}
  &= t^2 + 1,\\
  \twistedAlexAd{E_K}{\rho}
  &=
  (t-1)(t^2 + t + 1)\\
  &= (t-1)\Delta_K(-t).
\end{align*}

\subsection{Figure eight knot}
Let $K$ be the figure eight knot.
We consider the free Seifert surface illustrated as in Figure~\ref{fig:SeifertSurfaceFigure8}.
The Lin presentation associated to this Seifert surface is expressed as 
\begin{equation}
  \label{eqn:LinPresentationFigure8}
  \knotgroup
  =
  \langle
  x_1, x_2, \mu \,|\,
  \mu x_1 \mu^{-1} = x_1 x_2^{-1},\,
  \mu x_2 x_1 \mu^{-1} = x_2
  \rangle.
\end{equation}

\begin{figure}[!ht]
  \begin{center}
    \includegraphics[scale=.5]{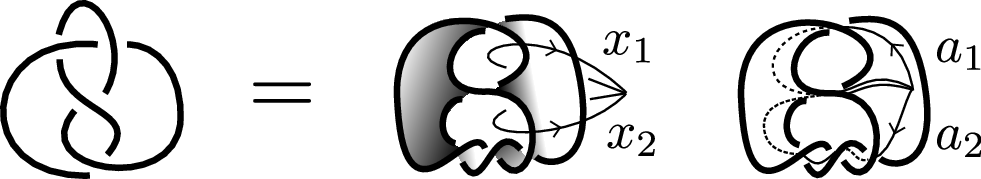}
  \end{center}
  \caption{A free Seifert surface of the figure eight knot}\label{fig:SeifertSurfaceFigure8}
\end{figure}

The Alexander polynomial $\Delta_K(t)$ is given by $t^2 -3t +1$.  The
knot determinant is $\Delta_K(-1) = 5$.  Since the number of conjugacy
classes of irreducible metabelian representations is given by
$(|\Delta_K(-1)| - 1)/2$, we have two irreducible metabelian
representations, up to conjugation.  By
Proposition~\ref{prop:explicitIrredMetabelian} and the relations in
the Lin presentation~\eqref{eqn:LinPresentationFigure8},
we have the following representatives $\rho_1$ and $\rho_2$ in the two conjugacy classes
of irreducible metabelian representations:
$$
\rho_k \co
x_1 \mapsto
\begin{pmatrix}
  \zeta_5^k & 0 \\
  0 & \zeta_5^{-k}
\end{pmatrix},\quad
x_2 \mapsto
\begin{pmatrix}
  \zeta_5^{2k} & 0 \\
  0 & \zeta_5^{-2k}
\end{pmatrix},\quad 
\mu  \mapsto
\begin{pmatrix}
  0 & 1 \\
  -1 & 0
\end{pmatrix}
\quad (k=1, 2)
$$
where $\zeta_5=\exp (2\pi\sqrt{-1}/5)$.
Then we have the same twisted Alexander polynomials $\twistedAlex{E_K}{\rho_k}$ and $\twistedAlexAd{E_K}{\rho_k}$ 
for both $k=1, 2$.
They are expressed, up to a factor $\pm t^n (n \in \Z)$, as
\begin{align*}
  \twistedAlex{E_K}{\rho_k}
  &= t^2 + 1,\\
  \twistedAlexAd{E_K}{\rho_k}
  &=
  (t-1)(t^2 + 3t + 1)\\
  &= (t-1)\Delta_K(-t).
\end{align*}

\subsection{$5_1$ knot ($(2, 5)$--torus knot)}
Let $K$ be the $5_1$ knot as in the table of Rolfsen.
This knot is the torus knot with type $(2, 5)$.
We consider the free Seifert surface illustrated as in Figure~\ref{fig:SeifertSurface5_1}.
The Lin presentation associated to this Seifert surface is expressed as 
\begin{equation}
  \label{eqn:LinPresentation5_1}
  \knotgroup
  =
  \left\langle
    x_1, x_2, x_3, x_4, \mu \, \left|\,
      \begin{array}{cc}
        \mu x_1 \mu^{-1} = x_1 x_2^{-1},&
        \mu x_2^{-1} x_1 \mu^{-1} = x_3 x_2^{-1},\\
        \mu x_2^{-1} x_3 \mu^{-1} = x_3 x_4^{-1},&
        \mu x_4^{-1} x_3 \mu^{-1} = x_4^{-1}
      \end{array}
    \right.
  \right\rangle.
\end{equation}

\begin{figure}[!ht]
  \begin{center}
    \includegraphics[scale=.5]{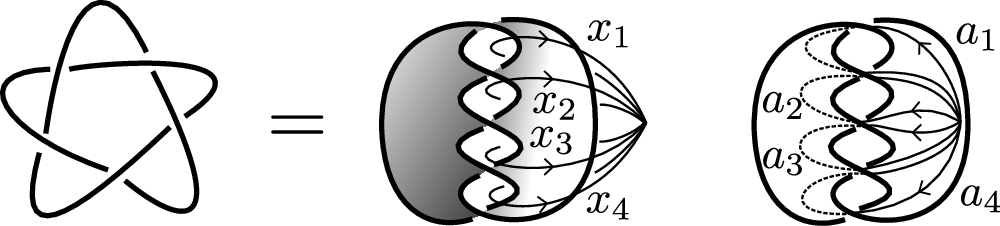}
  \end{center}
  \caption{A free Seifert surface of $5_1$ knot}\label{fig:SeifertSurface5_1}
\end{figure}

The Alexander polynomial $\Delta_K(t)$ is given by $t^4 -t^3 +t^2 -t+1$.
The knot determinant is $\Delta_K(-1) = 5$.
Since the number of conjugacy classes of irreducible metabelian representations is given by $(|\Delta_K(-1)|-1)/2$,
we have two irreducible metabelian representations, up to conjugation.
Proposition~\ref{prop:explicitIrredMetabelian} and the relations in
the Lin presentation~$(\ref{eqn:LinPresentationTrefoil})$,
we have the following representatives $\rho_1$ and $\rho_2$ in those two conjugacy classes
of irreducible metabelian representations:
$$
\rho_k \co
x_i \mapsto
\begin{pmatrix}
  \big(\zeta_5^k \big)^i & 0\\
  0 & \big(\zeta_5^{-k}\big)^i
\end{pmatrix},\quad
\mu \mapsto 
\begin{pmatrix}
  0 & 1 \\
  -1 & 0
\end{pmatrix}
\quad
(k=1, 2)
$$
where $\zeta_5 = \exp (2\pi\sqrt{-1}/5)$.  Then we have the following
list of the twisted Alexander polynomial $\twistedAlex{E_K}{\rho_k}$
and $\twistedAlexAd{E_K}{\rho_k}$ for $k=1, 2$:
\begin{center}
  \renewcommand{\arraystretch}{1.5}
  \begin{tabular}{c|c|c}
    & $\twistedAlexAd{\knotexterior}{\rho_k}$ & $\twistedAlex{\knotexterior}{\rho_k}$ \\
    \hline
    $k=1$ & $(t-1)\Delta_K(-t)(t^2 - \zeta_5)(t^2 - \zeta_5^{-1})$ & $(t^2+1)(t^2 + \zeta_5^2)(t^2 + \zeta_5^{-2})$ \\
    $k=2$ & $(t-1)\Delta_K(-t)(t^2 - \zeta_5^2)(t^2 - \zeta_5^{-2})$ & $(t^2+1)(t^2 + \zeta_5)(t^2 + \zeta_5^{-1})$
  \end{tabular}
  \renewcommand{\arraystretch}{1}
\end{center}

\subsection{$5_2$ knot}
Let $K$ be $5_2$ knot in the table of Rolfsen.
We consider the free Seifert surface illustrated as in Figure~\ref{fig:SeifertSurface5_2}.
The Lin presentation associated to this Seifert surface is expressed as 
\begin{equation}
  \label{eqn:LinPresentation5_2}
  \knotgroup
  =
  \langle
  x_1, x_2, \mu \,|\,
  \mu x_1 \mu^{-1} = x_1 x_2^{-1},\,
  \mu x_2^{-2} x_1 \mu^{-1} = x_2^{-2}
  \rangle.
\end{equation}

\begin{figure}[!ht]
  \begin{center}
    \includegraphics[scale=.5]{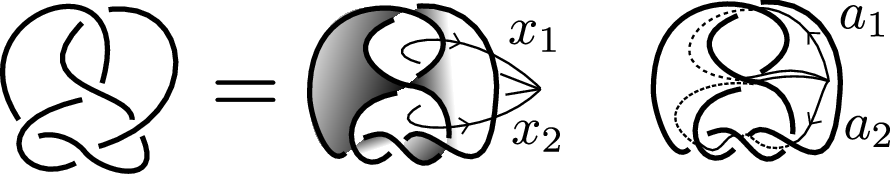}
  \end{center}
  \caption{A free Seifert surface of $5_2$ knot}\label{fig:SeifertSurface5_2}
\end{figure}

The Alexander polynomial $\Delta_K(t)$ is given by $2t^2 -3t +2$.  The
knot determinant is $\Delta_K(-1) = 7$.  Since the number of conjugacy
classes of irreducible metabelian representations is given by
$(|\Delta_K(-1)| - 1)/2$, we have three irreducible metabelian
representations, up to conjugation.  By
Proposition~\ref{prop:explicitIrredMetabelian} and the relations in
the Lin presentation~$(\ref{eqn:LinPresentationTrefoil})$,
we have the following representatives $\rho_1$, $\rho_2$ and $\rho_3$ in those three conjugacy classes
of irreducible metabelian representations:
$$
\rho_k:
x_1 \mapsto
\begin{pmatrix}
  \zeta_7^k & 0 \\
  0 & \zeta_7^{-k}
\end{pmatrix},\quad
x_2 \mapsto
\begin{pmatrix}
  \zeta_7^{2k} & 0 \\
  0 & \zeta_7^{-2k}
\end{pmatrix},\quad 
\mu  \mapsto
\begin{pmatrix}
  0 & 1 \\
  -1 & 0
\end{pmatrix}
\quad (k=1, 2, 3)
$$
where $\zeta_7 = \exp (2\pi\sqrt{-1}/7)$.
Then we have the following list of the twisted Alexander polynomial $\twistedAlex{E_K}{\rho_k}$
and $\twistedAlexAd{E_K}{\rho_k}$ 
for $k=1, 2, 3$:
\begin{center}
  \renewcommand{\arraystretch}{1.5}
  \begin{tabular}{c|c|c}
    & $\twistedAlexAd{\knotexterior}{\rho_k}$ & $\twistedAlex{\knotexterior}{\rho_k}$ \\
    \hline
    $k=1$ & $(t-1)(\zeta_7^3 + \zeta_7^{-3} + 2)\Delta_K(-t)$ & $(\zeta_7^2 + \zeta_7^{-2} + 2)(t^2+1)$ \\
    $k=2$ & $(t-1)(\zeta_7 + \zeta_7^{-1} + 2)\Delta_K(-t)$ & $(\zeta_7^3 + \zeta_7^{-3} + 2)(t^2+1)$ \\
    $k=3$ & $(t-1)(\zeta_7^2 + \zeta_7^{-2} + 2)\Delta_K(-t)$ & $(\zeta_7 + \zeta_7^{-1} + 2)(t^2+1)$ 
  \end{tabular}
  \renewcommand{\arraystretch}{1}
\end{center}
where $\zeta_7 = \exp(2\pi\sqrt{-1}/7)$.

\section*{Acknowledgments}
This research was supported by 
Research Fellowships of the Japan Society for the Promotion of Science for Young Scientists.
The author would like to express his thanks to Stephan Friedl for his useful comments at RIMS seminar 
\lq\lq Twisted topological invariants and topology of low--dimensional manifolds\rq\rq.
The author wishes to express his thanks to Kunio Murasugi for his helpful comments and suggestions.
The author also would like to express his gratitude to the referee for
his or her careful reading and many comments to improve this paper.
\bibliographystyle{amsalpha}
\bibliography{twistedAlexanderMetabelian}

\end{document}